\pdfoutput=1
\documentclass[11pt,a4paper]{amsart}

\usepackage[utf8]{inputenc}

\usepackage[left=1in,right=1in,top=1in,bottom=1in]{geometry}

\usepackage{amssymb}
\usepackage{amsmath}
\usepackage{amsthm}
\usepackage{thmtools}
\allowdisplaybreaks

\usepackage{stmaryrd}

\usepackage{tikz}
\usepackage{tikz-cd}

\usepackage{enumitem}

\newlist{enumalpha}{enumerate}{1}
\setlist[enumalpha, 1]{
label=(\alph*)
}

\usepackage{mathtools}

\usepackage[hypertexnames=false]{hyperref}

\usepackage[english,noabbrev,sort]{cleveref}

\let\C\undefined
\usepackage{my-math-definitions} 

\newtheorem{theorem}{Theorem}[section]
\newtheorem{lemma}[theorem]{Lemma}
\newtheorem{corollary}[theorem]{Corollary}
\newtheorem{definition}[theorem]{Definition}
\theoremstyle{remark}
\newtheorem{example}[theorem]{Example}
\newtheorem{remark}[theorem]{Remark}

\numberwithin{equation}{section}

\DeclareMathOperator{\jump}{jmp}
\DeclareMathOperator{\lastjump}{lstjmp}
\DeclareMathOperator{\asc}{asc}
\DeclareMathOperator{\inv}{inv}
\DeclareMathOperator{\sep}{sep}
\newcommand{\Fasc}{F^{\asc}}
\newcommand{\Fcond}{F^{\cond}}
\newcommand{\Fdisc}{F^{\disc}}
\renewcommand{\Finv}{F^{\inv}}

\author{Fabian Gundlach}
\address{Universität Paderborn, Fakultät EIM, Institut für Mathematik, Warburger Str.~100, 33098 Paderborn, Germany.}
\email{fabian.gundlach@uni-paderborn.de}
\title{Counting abelian extensions by Artin--Schreier conductor}

\begin{document}

\begin{abstract}
Let $G$ be a finite abelian $p$-group. We count étale $G$-extensions of global rational function fields $\mathbb F_q(T)$ of characteristic $p$ by the degree of what we call their Artin--Schreier conductor. The corresponding (ordinary) generating function turns out to be rational. This gives an exact answer to the counting problem, and seems to beg for a geometric interpretation.

This is in contrast with the generating functions for the ordinary conductor (from class field theory) and the discriminant, which in general have no meromorphic continuation to the entire complex plane.

\end{abstract}

\maketitle

\setcounter{tocdepth}{1}
\tableofcontents

\section{Introduction}

Let $G$ be a finite abelian group and let $K$ be a local or global field of characteristic~$p$. Let $K^{\sep}$ be the separable closure of $K$.

In this paper, the term \emph{$G$-extension} will refer to what is often called an étale $G$-extension, a sub-$G$-extension, or a $G$-torsor. (See \Cref{subsn:gextensions} for the definition.) Since $G$ is abelian, isomorphism classes of $G$-extensions are in bijection with continuous (not necessarily surjective) homomorphisms from the absolute Galois group of $K$ to $G$.

There are several natural divisors one can associate to an (étale) $G$-extension $L$ of $K$, such as the discriminant divisor $\disc(L)$, or the conductor (in the sense of class field theory) $\cond(L)$. It is natural to ask for the number of $G$-extensions of $K$ such that the discriminant divisor (or the conductor divisor) has a particular degree $n$.

If the characteristic $p$ does not divide the order of $G$, then places of $K$ can only ramify tamely in $L$. In this case, if $K$ is a local field, there are only finitely many $G$-extensions of $K$ altogether. If $K$ is a global field, Wright \cite{wright-counting-abelian-extensions} gave an asymptotic formula for the number of extensions~$L$ with $\deg(\disc(L)) \leq n$ as $n$ goes to infinity.

Instead, assume from now on that $G$ is an abelian $p$-group. In this ``totally wild'' (nowhere tamely ramified) case, asymptotics for the number of $G$-extensions have been computed under the following circumstances:
\begin{itemize}
\item For local fields, counting by conductor. (See Klüners and Müller \cite{klueners-mueller-conductor-density-local-fields}.)
\item For global fields, counting by conductor. (See Lagemann \cite{lagemann-artin-schreier}, \cite{lagemann-artin-schreier-witt}.)
\item For local fields, counting by discriminant. (See Lagemann \cite{lagemann-thesis}.)
\item For local and global fields, counting by discriminant, assuming the group $G$ is elementary abelian. (See Potthast \cite{potthast-elementary-abelian}.)
\end{itemize}

In this paper, we associate to any $G$-extension $L$ another effective divisor, which we call the \emph{Artin--Schreier conductor} $\asc(L)$. (See \Cref{def:asc}. The Artin--Schreier conductor divisor can be obtained from the ordinary conductor divisor by simply reducing the multiplicity of each ramified point by $1$.)
We explicitly compute the corresponding (ordinary) generating function
\[
\Fasc_K(X)
:= \frac1{|G|}\sum_{\substack{L\textnormal{ (étale) $G$-extension of }K\\\textnormal{up to isomorphism}}} X^{\deg(\asc(L))}
= \frac1{|G|}\sum_{n\geq0} |\{L : \deg(\asc(L)) = n\}| \cdot X^n
\in \Z\llbracket X\rrbracket
\]
for rational function fields $K$:
\begin{theorem}[see \Cref{thm:main}\ref{item:main-global-generating-function}]
\label{thm:main-intro}
Let $G=\prod_{e\geq1}C_{p^e}^{m_e}$ be a finite abelian $p$-group. Define
\[
\textstyle
c_0 := 0,
\qquad
c_1 := \sum_{e\geq1} m_e = \dim_{\F_p}(G[p]),
\qquad
c_{i+1} := c_i - m_i p^{-i}\textnormal{ for }i\geq1,
\]
or, equivalently, for $i\geq1$,
\[
c_i
:= \sum_{e=1}^\infty m_e - \sum_{e=1}^{i-1}m_ep^{-e}
= \sum_{j=1}^{i-1}(p-1)p^{-j}r_j + p^{1-i}r_i,
\qquad\textnormal{where } r_i := \dim_{\F_p}(G[p^i]/G[p^{i-1}]).
\]
For any rational function field $K=\F_q(T)$ of characteristic~$p$, we have
\[
\Fasc_K(X)
= \prod_{i=0}^\infty\frac{Z_K\big((q^{c_{i+1}}X)^{p^i}\big)}{Z_K\big((q^{c_i}X)^{p^i}\big)},
\]
where
\[
Z_K(X) = \frac{1}{(1-X)(1-qX)}
\]
is the Hasse--Weil zeta function of~$\vP^1_{\F_q}$.
\end{theorem}
Note that $m_i=0$ and hence $c_{i+1}=c_i$ for all sufficiently large $i$. Thus, all but finitely many factors in the infinite product are $1$. In particular, $\Fasc_K(X)$ is a rational function, so one can obtain not only an \emph{asymptotic} formula for the number $a_n$ of $G$-extensions $L$ of $K$ with $\deg(\asc(L))=n$ (see \Cref{cor:poles-and-asymptotics}\ref{item:poles-and-asymptotics-global}), but a recurrence relation (and in fact an \emph{exact} formula) for the numbers~$a_n$.

Using inclusion--exclusion, one can show that the generating functions counting only those $G$-extensions of $K=\F_q(T)$ which are fields (or, equivalently, counting surjective continuous group homomorphisms from the absolute Galois group of $K$ to $G$) are still rational. (See \Cref{rmk:only-fields}.)

The rationality of $\Fasc_K(X)$ could be a hint that there is a geometric interpretation of our result.
(See \cite{pries-zhu-p-rank-stratification} and \cite{dang-hippold-cyclic-covers-positive-characteristic} for a description of the moduli spaces of $C_{p^e}$-extensions with a given ramification filtration.)

It might be worth pointing out that the asymptotics are quite a bit ``simpler'' when counting extensions of $\F_q(T)$ by Artin--Schreier conductor than when counting by ordinary conductor or discriminant. (Compare \Cref{cor:poles-and-asymptotics}\ref{item:poles-and-asymptotics-global} to \cite[Theorem~1.2]{lagemann-artin-schreier-witt} and \cite[Theorem~5.30]{potthast-elementary-abelian}, respectively.)

The rationality of the generating function $\Fasc_K(X)$ is rather surprising:
generating functions naturally appearing in arithmetic statistics are rarely known (or even expected) to even admit a meromorphic continuation to the entire complex plane. Over function fields $K=\F_q(T)$, the author is aware of the following cases:
\begin{itemize}
\item
	The generating function counting $C_3$-extensions of $K=\F_3(T)$ by discriminant or by ordinary conductor does not have a meromorphic continuation to the entire complex plane.
	(See \Cref{sn:other-invariants}.)
\item
	If $G\neq1$ is a finite abelian group and the characteristic of $K$ does not divide $|G|$ (unlike what we assume throughout the rest of this paper), one can show that the generating function counting $G$-extensions of $K$ by discriminant, by ordinary conductor, or by sum of ramified points only has a meromorphic continuation to the entire complex plane if $G=C_2$.
\item
	The Shintani zeta function counting $\SL_2(\F_q[T])$-orbits of binary cubic forms with coefficients in $\F_q[T]$ by discriminant is rational. A similar statement holds for adelic Shintani zeta functions, which essentially allow imposing a finite number of congruence conditions on the cubic forms. (See \cite{datskovsky-adelic-zeta-over-function-field}.) However, despite Levi's parameterization of cubic extensions by cubic forms, it seems doubtful whether the generating function counting cubic extensions of $\F_q(T)$ by discriminant would be rational (or even have a meromorphic continuation to the entire complex plane), essentially because one would need to impose infinitely many congruence conditions.
\end{itemize}

\textbf{Acknowledgements.}
This work was supported by the Deutsche Forschungsgemeinschaft (DFG, German Research Foundation) --- Project-ID 491392403 --- TRR 358 (project A4).

The author would like to thank Kiran Kedlaya, Jürgen Klüners, Nicolas Potthast, and Béranger Seguin for helpful discussions and comments on an earlier draft.
Moreover, the author is grateful to the anonymous referees for their careful reading and helpful comments and corrections.

\section{Definitions}
\label{sn:definitions}

\subsection{$G$-extensions}
\label{subsn:gextensions}

Let $G$ be a finite group and let $K$ be a field.

A \emph{$G$-extension} $L$ of $K$ is an étale $K$-algebra~$L$ together with an action of $G$ such that there is a $G$-equivariant $K^{\sep}$-algebra isomorphism between $L\otimes_K K^{\sep}$ and the ring of maps $G\to K^{\sep}$, on which $G$ acts by $(g.f)(h) = f(hg)$. An \emph{isomorphism} of $G$-extensions is a $G$-equivariant $K$-algebra isomorphism.

Isomorphism classes of $G$-extensions of $K$ are in bijection with $G$-conjugacy classes of continuous homomorphisms from the absolute Galois group $\Gamma_K$ to $G$.

Any Galois extension $L/K$ together with an isomorphism $i:\Gal(L/K)\stackrel\sim\to G$ is a $G$-extension, and the Galois extensions correspond to the surjective homomorphisms $\Gamma_K\to G$.

\subsection{Invariants}

Let $G$ be a finite abelian $p$-group and let $K$ be a local or global field of characteristic $p$ with field of constants $\F_q$. Let $M_K$ be the set of places of $K$. We denote the degree of a place $P$ by $\deg(P)$. (Any local field has only one place, and its degree is $1$.) For any place $P$, let $K_P$ be the corresponding completion of $K$, let $\O_P$ be its valuation ring, and let $Q_P = q^{\deg(P)}$ be the size of its residue field. We denote the $k$-th higher unit group by $U_P^k$.

For any (étale) $G$-extension $L$ of $K$ and any place $P$ of $K$, let $I^t_P(L|K) \subseteq G$ be the $t$-th ramification group in upper numbering. (Since $G$ is abelian, this group is independent of the choice of prime of $L$ above $P$.)

\begin{definition}
\label{def:asc}
The \emph{Artin--Schreier conductor} of $L$ is the divisor
\begin{align*}
\asc(L) := \sum_{P\in M_K} \lastjump_P(L) \cdot P,
\end{align*}
where
\begin{align*}
\lastjump_P(L)
:= \inf\{t\in\R_{\geq0} \mid I_P^t(L|K) = 1\}
\end{align*}
is the last jump (in upper numbering) in the ramification filtration associated to $L$.
\end{definition}

\begin{remark}
The last jump satisfies the following properties:
\begin{enumalpha}
\item By the Hasse--Arf theorem, $\lastjump_P(L)$ is an integer.
\item If $P$ is unramified in $L$, then $\lastjump_P(L)=0$. In particular, the above formula actually defines a divisor (with finite support). (For the ramified places, this number $\lastjump_P(L)$ agrees with what Potthast calls the Artin--Schreier conductor in \cite[Proposition~3.3]{potthast-elementary-abelian}.)
\item Conversely, if $P$ is ramified, it must be wildly ramified, so $\lastjump_P(L) > 0$. In this case, $I_P^{\lastjump_P(L)}(L|K) \neq 1$.
\end{enumalpha}
\end{remark}

The Artin--Schreier conductor is closely related to the \emph{ordinary conductor} from class field theory, which at the ramified places is off by one from the Artin--Schreier conductor:
\[
\cond(L)
:= \sum_{P\in M_K\textnormal{ ramified in }L} (\lastjump_P(L)+1) \cdot P.
\]

We write $\disc(L)$ for the relative \emph{discriminant} divisor of $L|K$.

\subsection{Generating functions}

For any $\inv\in\{\asc,\cond,\disc\}$, we define
\[
\Finv_K(X)
:= \frac1{|G|}\sum_{\substack{L\textnormal{ (étale) $G$-extension of }K\\\textnormal{up to isomorphism}}} X^{\deg(\inv(L))}
= \frac1{|G|}\sum_{n\geq0} |\{L : \deg(\inv(L)) = n\}| \cdot X^n.
\]
For the local field $K_P$ and Artin--Schreier conductor, this means
\[
\Fasc_{K_P}(X) = \frac{1}{|G|} \sum_{\substack{L\textnormal{ (étale) $G$-extension of }K_P\\\textnormal{up to isomorphism}}} X^{\lastjump_P(L)} = \frac{1}{|G|}\sum_{n\geq0}|\{L:\lastjump_P(L)=n\}|\cdot X^n.
\]

\section{Class field theory}

To compute the generating function $\Fasc_K(X)$, we use the same general strategy as Wright in \cite{wright-counting-abelian-extensions}, combined with a description of the ramification filtration in terms of class field theory. (See also \cite{lagemann-artin-schreier-witt} and \cite{klueners-mueller-conductor-density-local-fields}.) Étale $G$-extensions of any field $K$ are naturally in bijection with continuous group homomorphisms (in the following just called \emph{maps}) $\varphi: \Gal(K^\ab|K) \ra G$.

From now on, let $K=\F_q(T)$ be a global rational function field of characteristic $p$.

Class field theory provides a description of the group $\Gal(K^\ab|K)$.
For any place $P\in M_K$, we have an Artin reciprocity isomorphism
\[
\theta_P : \widehat{K_P^\times} \stackrel\sim\longrightarrow \Gal(K_P^\ab|K_P^{}).
\]
Choosing a uniformizer $\pi\in K_P^\times$, we obtain an isomorphism $\O_P^\times\times\Z\ra K_P^\times$, $(u,n)\mapsto u\pi^n$, which induces an isomorphism
\[
\tau_P : \O_P^\times\times\widehat\Z \stackrel\sim\longrightarrow \widehat{K_P^\times}.
\]
The image of the unit group $\O_P^\times$ under the composition $\theta_P\circ\tau_P$ is the inertia group $I(K_P^\ab|K_P^{})$.
The image of the $k$-th higher unit group $U_P^k$ is the $k$-th higher ramification group $I^k(K_P^\ab|K_P^{})$.

Moreover, we have an Artin reciprocity isomorphism
\[
\theta_K : \widehat{\A_K^\times/K^\times} \stackrel\sim\longrightarrow \Gal(K^\ab|K).
\]
Choosing a place $P$ of degree $1$ and a uniformizer $\pi\in K_P^\times$ of this place, we obtain a map $\Z\ra\A_K^\times$, $1\mapsto(\dots,1,\pi,1,\dots)$. Combining this map with the embedding $\prod_P\O_P^\times\hookrightarrow\A_K^\times$ gives rise to an isomorphism
\[
\tau_K : \bigg[\Big(\prod_{P\in M_K}\O_P^\times\Big) / \F_q^\times\bigg] \times \widehat\Z \stackrel\sim\longrightarrow \widehat{\A_K^\times/K^\times}
\]
(The surjectivity of this map relies on the fact that the rational function field $\F_q(T)$ has trivial class group.)

For any place $P$, we obtain the following commutative diagram:
\begin{equation}
\label{eq:artin-reciprocity-compatibility}
\begin{tikzcd}
\O_P^\times \dar[hook] \rar[hook]{\tau_P} & \widehat{K_P^\times} \dar[hook] \rar{\theta_P} & \Gal(K_P^\ab|K_P^{}) \dar[hook] \\
\Big(\prod_{P\in M_K}\O_P^\times\Big) / \F_q^\times \rar[hook]{\tau_K} & \widehat{\A_K^\times/K^\times} \rar{\theta_K} & \Gal(K^\ab|K)
\end{tikzcd}
\end{equation}
The problem of computing the desired generating functions now simplifies as follows:

\begin{lemma}[Local generating functions]
\label{lem:local-generating-functions}
Let $P$ be any place of $K=\F_q(T)$.
\begin{enumalpha}
\item\label{item:local-generating-functions-extension}
Every map $\varphi:U_P^1\ra G$ has a unique extension to $\O_P^\times$.
\item\label{item:local-generating-functions-asc}
We have $\Fasc_{K_P}(X) = \sum_{\varphi:U_P^1\ra G} X^{\lastjump_P(\varphi)}$, where we let $\lastjump_P(\varphi)$ be the smallest integer $k\geq0$ such that $\varphi(U_P^{k+1})=1$.
\item\label{item:local-generating-functions-cond}
The generating functions for the Artin--Schreier conductor and for the ordinary conductor are related as follows:
\[
\Fcond_{K_P}(X) - 1 = X\cdot\left(\Fasc_{K_P}(X)-1\right).
\]
\end{enumalpha}
\end{lemma}
\begin{proof}
\begin{enumalpha}
\item
This follows from the natural isomorphism $\F_{Q_P}^\times\times U_P^1\stackrel\sim\ra\O_P^\times$ since the order $Q_P-1$ of $\F_{Q_P}^\times$ is coprime to the order of $G$.
\item\label{item:local-generating-functions-asc-proof}
By class field theory, we have a bijection between (étale) $G$-extensions $L$ of $K_P$ and maps $\varphi:\O_P^\times\times\widehat\Z\ra G$. There are exactly $|G|$ maps (continuous homomorphisms) $\widehat\Z\ra G$, so each map $\varphi:U_P^1\ra G$ has exactly $|G|$ extensions to $\O_P^\times\times\widehat\Z$. For $k\geq0$, we have the following equivalences:
\begin{align*}
&\lastjump_P(L) \leq k
\iff I_P^{k+1}(L|K) = 1
\iff \varphi\left((\theta_P\circ\tau_P)^{-1}(I^{k+1}(K_P^\ab|K_P^{}))\right) = 1 \\
&\iff \varphi(U_P^{k+1}) = 1
\iff \lastjump_P(\varphi) \leq k.
\end{align*}
Hence,
\[
\Fasc_{K_P}(X)
= \frac1{|G|}\sum_L X^{\lastjump_P(L)}
= \sum_{\varphi:U_P^1\ra G} X^{\lastjump_P(\varphi)}.
\]
\item
Let $L$ correspond to $\varphi$ as in the proof of \ref{item:local-generating-functions-asc-proof}. The local extension $L|K_P$ is unramified if and only if $I_P(L|K)=1$, which is equivalent to $\varphi(\O_P^\times)=1$. This is equivalent to the restriction $\varphi:U_P^1\ra G$ being the trivial map. Thus, subtracting the trivial map,
\[
\Fcond_{K_P}(X) - 1
= \sum_{\varphi:U_P^1\ra G\textnormal{ non-trivial}} X^{\lastjump_P(\varphi)+1}
\]
and
\[
\Fasc_{K_P}(X) - 1
= \sum_{\varphi:U_P^1\ra G\textnormal{ non-trivial}} X^{\lastjump_P(\varphi)}.
\qedhere
\]
\end{enumalpha}
\end{proof}

\begin{lemma}[Local-global principle]
\label{lem:local-global-principle}
Let $K=\F_q(T)$. For any invariant $\inv\in\{\asc,\cond,\disc\}$, we have
\[
\Finv_K(X) = \prod_{P\in M_K} \Finv_{K_P}\left(X^{\deg(P)}\right).
\]
\end{lemma}
\begin{proof}
Each tuple $(\varphi_P)_{P\in M_P}$ of maps $\varphi_P : \O_P^\times \ra G$, almost all of which are trivial, gives rise to exactly $|G|$ maps $\varphi: (\prod_P\O_P^\times) / \F_q^\times \times \widehat\Z \ra G$ since the order of $\F_q^\times$ is coprime to the order of $G$ and since there are exactly $|G|$ maps $\widehat\Z\ra G$. Using the commutative diagram (\ref{eq:artin-reciprocity-compatibility}), we obtain:
\begin{align*}
\Fasc_K(X)
&= \frac{1}{|G|}\sum_L X^{\deg(\asc(L))}
= \frac{1}{|G|}\sum_L \prod_P X^{\lastjump_P(L)\deg(P)} \\
&= \prod_P \sum_{\varphi_P} \left(X^{\deg(P)}\right)^{\lastjump_P(\varphi_P)}
= \prod_P \Fasc_{K_P}\left(X^{\deg(P)}\right).
\end{align*}
For the other two invariants $\inv=\cond$ and $\inv=\disc$, the claim follows in the same way since these invariants can also be computed locally from the images of the higher ramification groups under the maps $\varphi_P$.
\end{proof}

\begin{remark}[{cf.\ \cite[page 28f]{wright-counting-abelian-extensions}}]
\label{rmk:only-fields}
The (étale) $G$-extension $L$ corresponding to a map $\varphi:\Gal(K^\ab|K)\ra G$ is a field if and only if $\varphi$ is surjective. Using inclusion--exclusion over the image of $\varphi$, one can therefore write the generating function $F^{\inv}_{K,\textnormal{field}}(X)$ counting only field $G$-extensions of $K$ as a linear combination of the generating functions counting (étale) $G'$-extensions of $K$ for all subgroups $G'\subseteq G$. In particular, if all the generating functions counting (étale) extensions are rational functions, then so is the generating function counting field extensions.
\end{remark}

\section{Artin--Schreier conductors}

Our goal is to compute the local generating functions $\Fasc_{K_P}(X)$ using \Cref{lem:local-generating-functions}\ref{item:local-generating-functions-asc}. To this end, we first describe the structure of $U_P^1$ as an abstract group:

\begin{lemma}[{cf.\ \cite[Proposition II.5.7(ii)]{neukirch-algebraic-number-theory}}]
\label{lem:unit-group-structure}
Let $Q_P = q^{\deg(P)} = p^d$ be the size of the residue field of $\O_P$. Consider the non-decreasing function $\nu:\R_{>0}\ra\Z_{\geq0}$ given by
\[
\nu(x) := \min\{k\in\Z_{\geq0} \mid p^k \geq x\} = |\{j\in\Z_{\geq0} \mid p^j < x\}|.
\]
There is an isomorphism (of topological groups)
\[
\alpha : \prod_{\substack{i\geq1:\\p\nmid i}} \Z_p^d \stackrel\sim\longrightarrow U_P^1
\]
such that the preimage under $\alpha$ of $U_P^k \subseteq U_P^1$ is the subgroup
\[
\Gamma_P^k := \prod_{\substack{i\geq1:\\p\nmid i}} p^{\nu(k/i)}\Z_p^d.
\]
\end{lemma}
\begin{proof}
Let $t$ be a uniformizer of $K_P$, so that $K_P=\F_{Q_P}((t))$, and let $\omega_1,\dots,\omega_d\in\F_{Q_P}\subset K_P$ form an $\F_p$-basis of $\F_{Q_P}$. For any $a = (a_i)_i \in \Gamma_P^1 = \prod_{i\geq1:\ p\nmid i}\Z_p^d$ with $a_i=(a_{i1},\dots,a_{id})\in\Z_p^d$, we let
\[
\alpha(a)
:= \prod_{\substack{i\geq1:\\ p\nmid i}}
  \prod_{j=1}^d
    \left(1+\omega_jt^i\right)^{a_{ij}}.
\]
One can show that $\alpha$ is an isomorphism of topological groups and that $\alpha(a)-1$ has valuation $\min\{ip^{s_i}:i\geq1,\ p\nmid i,\ a_i\neq0\}$ if we write $a_i = p^{s_i} b_i$ with $b_i \in \Z_p^d \setminus p\Z_p^d$. (See the proof of \cite[Proposition~II.5.7(ii)]{neukirch-algebraic-number-theory}.) Hence, $\alpha(a)\in U_P^k$ if and only if $a\in\Gamma_P^k$.
\end{proof}

\begin{theorem}
\label{thm:main}
Let $G=\prod_{e\geq1} C_{p^e}^{m_e}$ and $c_0,c_1,\dots$ as in \Cref{thm:main-intro}, let $K=\F_q(T)$ be a rational function field of characteristic~$p$, and let $Z_K(X)$ be the Hasse--Weil zeta function of~$\vP^1_{\F_q}$.
\begin{enumalpha}
\item\label{item:main-local-count}
For any place $P$ of $K$ and any $k\geq0$, there are exactly $Q_P^{\tau(k)}$ maps $\varphi:U_P^1\ra G$ satisfying $\lastjump_P(\varphi)\leq k$, where
\[
\tau(k)
= \sum_{e\geq1}m_e\left(k - \left\lfloor\frac{k}{p^e}\right\rfloor\right).
\]
\item\label{item:main-local-generating-function}
The local generating function for a place $P$ is
\[
\Fasc_{K_P}(X)
= \prod_{i=0}^\infty \frac{1-(Q_P^{c_i}X)^{p^i}}{1-(Q_P^{c_{i+1}}X)^{p^i}}.
\]
\item\label{item:main-global-generating-function}
The global generating function is
\[
\Fasc_K(X)
= \prod_{i=0}^\infty\frac{Z_K\big((q^{c_{i+1}}X)^{p^i}\big)}{Z_K\big((q^{c_i}X)^{p^i}\big)}.
\]
\end{enumalpha}
\end{theorem}
Note that $m_i=0$ and hence $c_{i+1}=c_i$ for all sufficiently large $i$. Thus, all but finitely many factors in \ref{item:main-local-generating-function} and \ref{item:main-global-generating-function} are $1$.
\begin{proof}
\begin{enumalpha}
\item
Recalling the definition of $\lastjump_P(\varphi)$ in \Cref{lem:local-generating-functions}\ref{item:local-generating-functions-asc} and using the isomorphism $\alpha:\Gamma_P^1\ra U_P^1$ from \Cref{lem:unit-group-structure}, we see that the maps $\varphi:U_P^1\ra G$ that satisfy $\lastjump_P(\varphi)\leq k$ correspond to the maps $\varphi:\Gamma_P^1\ra G$ that satisfy $\varphi(\Gamma_P^{k+1})=1$, i.e., to the maps $\varphi: \Gamma_P^1/\Gamma_P^{k+1}\ra G$. By definition,
\[
\Gamma_P^1/\Gamma_P^{k+1}
\cong \prod_{\substack{i\geq1:\\p\nmid i}} \left(\Z/p^{\nu((k+1)/i)}\Z\right)^d.
\]
Hence,
\begin{align*}
|\{\varphi:U_P^1\ra G : \lastjump_P(\varphi)\leq k\}|
&= \prod_{\substack{i\geq1:\\p\nmid i}} |\Hom(\Z/p^{\nu((k+1)/i)}\Z, G)|^d \\
&= \prod_{\substack{i\geq1:\\p\nmid i}} \left|G\left[p^{\nu((k+1)/i)}\right]\right|^d.
\end{align*}
For any $r\geq0$, we have $|G[p^r]| = p^{\sum_{e\geq1}m_e\cdot\min(e,r)}$. Therefore, the number of $\varphi$ with $\lastjump_P(\varphi)\leq k$ is $p^{\tau(k)d}=Q_P^{\tau(k)}$ for
\begin{align*}
\tau(k)
&:= \sum_{\substack{i\geq1:\\p\nmid i}}\sum_{e\geq1}m_e\cdot\min\left(e,\nu\left(\frac{k+1}{i}\right)\right) \\
&= \sum_{e\geq1}m_e\cdot|\{(i,j) : i \geq 1,\ p\nmid i,\ 0\leq j<e,\ ip^j \leq k\}| \\
&\stackrel{\mathclap{(x:=ip^j)}}=\quad \sum_{e\geq1}m_e\cdot|\{1\leq x\leq k : p^e \nmid x\}| 
= \sum_{e\geq1}m_e\left(k - \left\lfloor\frac{k}{p^e}\right\rfloor\right).
\end{align*}
\item
With \Cref{lem:local-generating-functions}\ref{item:local-generating-functions-asc}, we obtain
\begin{align*}
&\Fasc_{K_P}(X)
=\sum_{\varphi:U_P^1\ra G} X^{\lastjump_P(\varphi)}
= \sum_{k\geq0} |\{\varphi : \lastjump_P(\varphi)=k\}| \cdot X^k \\
&= (1-X)\sum_{k\geq0} |\{\varphi : \lastjump_P(\varphi)\leq k\}| \cdot X^k 
= (1-X)\sum_{k\geq0}Q_P^{\tau(k)}X^k.
\end{align*}
Any integer $k\geq0$ can be uniquely written in base $p$ as $k=\sum_{i\geq 0} v_ip^i$ with integers $v_0,v_1,\dots\in\{0,\dots,p-1\}$, all but finitely many of which are $0$. We have
\begin{align*}
\tau\left(\sum_{i\geq0} v_ip^i\right)
&= \sum_{e\geq1} m_e\left(\sum_{i\geq0} v_ip^i - \sum_{i\geq e} v_ip^{i-e}\right) \\
&= \sum_{i\geq0} v_i p^i \cdot \left( \sum_{e=1}^\infty m_e - \sum_{e=1}^i m_e p^{-e} \right)
= \sum_{i\geq0} v_i p^i c_{i+1}.
\end{align*}
Thus,
\begin{align*}
\sum_{k\geq0} Q_P^{\tau(k)} X^k
&= \sum_{v_0,v_1,\dots} \prod_{i\geq0} Q_P^{v_ip^ic_{i+1}} X^{v_i p^i}
= \prod_{i\geq0} \sum_{v_i=0}^{p-1} Q_P^{v_i p^i c_{i+1}} X^{v_i p^i} \\
&= \prod_{i\geq0} \frac{1 - (Q_P^{c_{i+1}} X)^{p^{i+1}}}{1 - (Q_P^{c_{i+1}} X)^{p^i}}
= \frac{\prod_{i\geq1}\big(1 - (Q_P^{c_i} X)^{p^i}\big)}{\prod_{i\geq0}\big(1 - (Q_P^{c_{i+1}} X)^{p^i}\big)}.
\end{align*}
Since $c_0=0$, we obtain
\[
\Fasc_{K_P}(X)
= (1-X) \sum_{k\geq0} Q_P^{\tau(k)} X^k
= \frac{\prod_{i\geq0}\big(1 - (Q_P^{c_i}X)^{p^i}\big)}{\prod_{i\geq0}\big(1 - (Q_P^{c_{i+1}}X)^{p^i}\big)}.
\]
\item
The Hasse--Weil zeta function
\[
Z_K(X)
= \sum_{D\textnormal{ divisor on }\vP^1_{\F_q}} X^{\deg(D)}
= \frac{1}{(1-X)(1-qX)}
\]
of $\vP^1_{\F_q}$ can be written as an Euler product
\[
Z_K(X)
= \prod_{P\in M_K} \frac{1}{1-X^{\deg(P)}}.
\]
Combining \Cref{lem:local-global-principle} and part \ref{item:main-local-generating-function} with $Q_P = q^{\deg(P)}$, we obtain
\[
\Fasc_K(X)
= \prod_{P\in M_K} \prod_{i=0}^\infty \frac{1-(q^{c_i}X)^{p^i\deg(P)}}{1-(q^{c_{i+1}}X)^{p^i\deg(P)}}
= \prod_{i=0}^\infty\frac{Z_K\big((q^{c_{i+1}}X)^{p^i}\big)}{Z_K\big((q^{c_i}X)^{p^i}\big)}
\]
as claimed.
\qedhere
\end{enumalpha}
\end{proof}

\begin{example}
\label{ex:local-power-of-cyclic}
For the group $G=C_{p^e}^r = C_{p^e}\times\cdots\times C_{p^e}$, we have
\[
c_0=0,\qquad
c_1=\cdots=c_e=r,\qquad
c_{e+1}=c_{e+2}=\cdots=(1-p^{-e})r,
\]
so the local generating function for a place $P$ is
\[
\Fasc_{K_P}(X)
= \frac{
  \big(1-X\big)
  \big(1-Q_P^{p^er}X^{p^e}\big)
}{
  \big(1-Q_P^rX\big)
  \big(1-Q_P^{(p^e-1)r}X^{p^e}\big)
}
\]
and the global generating function is
\begin{align*}
\Fasc_K(X)
&= \frac{
  Z_K\big(q^rX\big)
  Z_K\big(q^{(p^e-1)r}X^{p^e}\big)
}{
  Z_K\big(X\big)
  Z_K\big(q^{p^er}X^{p^e}\big)
} \\
&= \frac{
  \big(1-X\big)
  \big(1-qX\big)
  \big(1-q^{p^er}X^{p^e}\big)
  \big(1-q^{p^er+1}X^{p^e}\big)
}{
  \big(1-q^rX\big)
  \big(1-q^{r+1}X\big)
  \big(1-q^{(p^e-1)r}X^{p^e}\big)
  \big(1-q^{(p^e-1)r+1}X^{p^e}\big)
}.
\end{align*}
For $r=1$, this simplifies to
\[
\Fasc_K(X)
= \frac{
  \big(1-X\big)
  \big(1-q^{p^e+1}X^{p^e}\big)
}{
  \big(1-q^2X\big)
  \big(1-q^{p^e-1}X^{p^e}\big)
}.
\]
\end{example}

\begin{remark}
From the partial fraction decomposition of the generating function $\Fasc_K(X)$, one can extract an exact formula for the number of (étale) $G$-extensions $L$ of $K$ with $\deg(\asc(L))=n$. (See for example the proof of \cite[Theorem~IV.9]{flajolet-sedgewick-analytic-combinatorics}.)
\end{remark}

\begin{corollary}
\label{cor:poles-and-asymptotics}
Assume $G$ is non-trivial. Then:
\begin{enumalpha}
\item
\label{item:poles-and-asymptotics-local}
The poles of the local generating function $\Fasc_{K_P}(X)$ are simple poles at the $p^t$ points $X\in\C$ with $(Q_P^a X)^{p^t}=1$, where $p^t$ is the exponent of $G$ and $a := \sum_{e=1}^\infty m_e(1-p^{-e}) = \sum_{j=1}^\infty (p-1)p^{-j}\dim_{\F_p}(G[p^j]/G[p^{j-1}])$.

By \Cref{lem:local-generating-functions}\ref{item:local-generating-functions-cond}, the local generating functions $\Fasc_{K_P}(X)$ and $\Fcond_{K_P}(X)$ have the same poles. The local asymptotics for the number of extensions of $K_P$ with a given degree of $\asc(L)$ or of $\cond(L)$ are of course essentially the same as their degrees are only off by $1$ (in the ramified case), so we recover the result of Klüners and Müller \cite[Theorem~1]{klueners-mueller-conductor-density-local-fields}.
\item
\label{item:poles-and-asymptotics-global}
The only innermost pole of the global generating function $\Fasc_K(X)$ is a simple pole at $X=q^{-a'}$, where $a' := 1+\dim_{\F_p}(G[p])$.
In particular, we have the following asymptotic statement for some constant $C>0$:
\[
|\{L\textnormal{ (étale) $G$-extension of }K : \deg(\asc(L)) = n\}|
\sim C q^{a' n}
\qquad\textnormal{ for }n\rightarrow\infty.
\]
\end{enumalpha}
\end{corollary}
\begin{proof}
We have $m_t\neq0$ but $m_{t+1}=m_{t+2}=\cdots=0$. Thus,
\[
c_1 \geq \cdots \geq c_t > c_{t+1} = c_{t+2} = \dots = a.
\]
\begin{enumalpha}
\item
We can write
\[
\Fasc_{K_P}(X)
= (1-X)
\cdot
\prod_{i=0}^{t-1} \frac{1 - (Q_P^{c_{i+1}} X)^{p^{i+1}}}{1 - (Q_P^{c_{i+1}} X)^{p^i}}
\cdot
\frac{1}{1 - (Q_P^{c_{t+1}} X)^{p^t}}.
\]
All factors except the last one are polynomials (geometric sums) and their zeros have absolute value different from $Q_P^{-a}$. The poles of the last factor are the points $X\in\C$ with $(Q_P^a X)^{p^t}=1$, which have absolute value $Q_P^{-a}$.
\item
Recall the expression
\[
\Fasc_K(X)
= \prod_{i=0}^\infty\frac{Z_K\big((q^{c_{i+1}}X)^{p^i}\big)}{Z_K\big((q^{c_i}X)^{p^i}\big)}.
\]
The pole at $q^{-a'} = q^{-1-c_1}$ comes from the factor $Z_K(q^{c_1}X)$ in the numerator. All poles and zeros of the other factors have larger absolute value.

The asymptotic statement on the coefficients of the power series $\Fasc_K(X)$ follows by subtracting the appropriate geometric series from $\Fasc_K(X)$ to remove the innermost pole. (See for example \cite[Theorem~IV.9]{flajolet-sedgewick-analytic-combinatorics}.)
\qedhere
\end{enumalpha}
\end{proof}

\begin{remark}
Let $G$ be a cyclic $p$-group, say $G=C_{p^e}$.
Instead of considering just the last jump in the ramification filtration, we could take all jumps into account: let
\[
\jump_{P,i}(\varphi) := \inf\{t\in\R_{\geq0} : \varphi(I^t(K_P^\ab|K_P^{})) \subseteq G[p^i]\}.
\]
By definition, $\jump_{P,i}(\varphi)=0$ for all $i\geq e$. Using the Hasse--Arf theorem and class field theory, one can show that the numbers $\jump_{P,i}(\varphi)$ are non-negative integers and that $\jump_{P,i}(\varphi) \geq p\cdot\jump_{P,i+1}(\varphi)$ for all $i$.
Using similar methods as above, one can then compute the multivariate generating function
\begin{align*}
F^{\jump}_{K_P}(X_0,\dots,X_{e-1})
&:= \sum_{\varphi: U_P^1 \ra G}
  \prod_{i=0}^{e-1} X_i^{\jump_{P,i}(\varphi)-p\cdot\jump_{P,i+1}(\varphi)} \\
&= \prod_{i=0}^{e-1}
  \frac{
    \big(1-Q_P^{p^i-1}X_i\big)
    \big(1-Q_P^{p^{i+1}}X_i^p\big)
  }{
    \big(1-Q_P^{p^i}X_i\big)
    \big(1-Q_P^{p^{i+1}-1}X_i^p\big)
  }.
\end{align*}
(Plugging in $(X_0,X_1,\dots,X_{e-1})=(X,X^p,\dots,X^{p^{e-1}})$, we recover the generating function $\Fasc_{K_P}(X)$ for the last jump $\lastjump_P(\varphi)=\jump_{P,0}(\varphi)$.)
The corresponding global multivariate generating function $F^{\jump}_K(X_0,\dots,X_{e-1}) = \prod_{P\in M_K} F^{\jump}_{K_P}(X_0^{\deg(P)},\dots,X_{e-1}^{\deg(P)})$ is again a finite product of values of the Hasse--Weil zeta function, and in particular rational.
\end{remark}

\section{Ordinary conductors and discriminants}
\label{sn:other-invariants}

While the local generating functions $\Finv_{K_P}(X)$ are rational functions for all invariants considered above (see \Cref{thm:main}\ref{item:main-local-generating-function} and \Cref{lem:local-generating-functions}\ref{item:local-generating-functions-cond} for $\inv=\asc$ and $\inv=\cond$; see \cite{potthast-elementary-abelian} for $\inv=\disc$ when $G$ is elementary abelian; a similar argument works for arbitrary abelian $p$-groups~$G$), this is in general not true for the global generating functions.

In this section, we show for the example $G=C_3$ that the global generating functions $\Fcond_K(X)$ and $\Fdisc_K(X)$ for ordinary conductors and discriminants (with $K=\F_q(T)$) cannot be continued meromorphically to the entire complex plane. (In particular, they are not rational functions.)

For $G=C_3$, combining \Cref{lem:local-generating-functions}\ref{item:local-generating-functions-cond} and \Cref{ex:local-power-of-cyclic}, we obtain the following local generating function for any place $P$:
\begin{equation}
\label{eq:C3-cond-local}
\Fcond_{K_P}(X)
= \frac{
  \big(1-X\big)
  \big(1+X-Q_PX-Q_P^2X^3\big)
}{
  \big(1-Q_PX\big)
  \big(1-Q_P^2X^3\big)
}
= \frac{
  \big(1-X\big)
  \big(1+X+Q_PX^2\big)
}{
  1-Q_P^2X^3
}.
\end{equation}
According to \Cref{lem:local-global-principle}, the global generating function is
\begin{equation}
\label{eq:C3-cond-global}
\Fcond_K(X)
= \prod_{P\in M_K} \Fcond_{K_P}\left(X^{\deg(P)}\right).
\end{equation}
\begin{theorem}
\label{thm:no-meromorphic-continuation}
Let $G = C_3$ and let $K=\F_q(T)$ be a rational function field of characteristic $p=3$. Then, the generating function $\Fcond_K(X)$ has a meromorphic continuation to the open disc of radius $q^{-1/2}$ centered at the origin. Every point on the boundary of this disc is an accumulation point of poles and zeros of the continuation. (Hence, the function cannot be meromorphically continued to any larger domain.)
\end{theorem}

\begin{remark}
For any $C_3$-extension $L$ of $K$, we have $\disc(L) = 2\cdot\cond(L)$ (as divisors!) by the conductor-discriminant formula. Hence, $\Fdisc_K(X) = \Fcond_K(X^2)$, which according to \Cref{thm:no-meromorphic-continuation} cannot be meromorphically continued beyond the open disc of radius $q^{-1/4}$.
\end{remark}

We will see that any meromorphic continuation of $\Fcond_K(X)$ beyond the disc of radius $q^{-1/2}$ is prevented by the ``stray'' Euler product $\prod_P (1+X+(qX^2)^{\deg(P)})$ on the right-hand side of \Cref{eq:C3-cond-global}.
The idea of the proof of \Cref{thm:no-meromorphic-continuation} is to approximate $\Fcond_K(X)$ by an infinite product of zeta functions (or rather its partial products) using the following lemma.
(This idea is of course not new. See for example \cite{estermann-functions-represented-by-dirichlet-series}.)
\begin{lemma}
\label{lem:local-approximation}
Let $G=C_3$ and let $P$ be a place of $K$. Let $A$ be a positive integer and let $\varepsilon=A^{-1}$. Define
\[
H_{P,A}(X)
:= \frac{
  \big(1-X\big)
  \big(1+Q_PX^2\big)
  \prod_{a=1}^A\big(1-Q_P^{a-1}X^{2a-1}\big)^{(-1)^a}
}{
  1-Q_P^2X^3
}.
\]
Then, the rational function $\Fcond_{K_P}(X) / H_{P,A}(X)$ satisfies
\[
|\Fcond_{K_P}(X) / H_{P,A}(X) - 1| \ll_A Q_P^{-1-\varepsilon} = q^{-(1+\varepsilon)\deg(P)}
\]
for $|X|\leq Q_P^{-(1+\varepsilon)/2}$ if $Q_P\gg_A 1$, and has no zeros and poles with $|X|<Q_P^{-1/2}$ for arbitrary $Q_P$. (The constants in $\ll_A$ and $\gg_A$ may depend on $A$, but not on the place $P$.)
\end{lemma}
\begin{proof}
Let $Q := Q_P$. By \Cref{eq:C3-cond-local}, we have
\[
\frac{\Fcond_{K_P}(X)}{H_{P,A}(X)}
= \frac{
  1+X+QX^2
}{
  \big(1+QX^2\big)
  \prod_{a=1}^A\big(1-Q^{a-1}X^{2a-1}\big)^{(-1)^a}
}.
\]
One can see from this formula that the quotient has no zeros or poles with $|X|<Q^{-1/2}$.
Using $|X|\leq Q^{-(1+\varepsilon)/2}$, which is equivalent to $|QX^2|\leq Q^{-\varepsilon}$, we compute
\begin{align*}
R
&:= -\log\left(1+X+QX^2\right)
= \sum_{a=1}^\infty \frac1a
  \big(\underbrace{
    -X-QX^2
  }_{
    \ll Q^{-\varepsilon}
  }\big)^a
= \sum_{a=1}^A \frac1a
  \left(
    -X-QX^2
  \right)^a
  + \O_A\left(Q^{-1-\varepsilon}\right), \\
S
&:= -\log\left(1+QX^2\right)
= \sum_{a=1}^\infty \frac1a
  \big(\underbrace{
    -QX^2
  }_{
    \ll Q^{-\varepsilon}
  }\big)^a
= \sum_{a=1}^A \frac1a
  \left(
    -QX^2
  \right)^a
  + \O_A\left(Q^{-1-\varepsilon}\right), \\
T_a
&:= - \log\left(1-Q^{a-1}X^{2a-1}\right)
= \sum_{k=1}^\infty \frac1k
  \big(\underbrace{
    Q^{a-1}X^{2a-1}
  }_{
    \ll Q^{-(1+\varepsilon)/2}
  }\big)^k
= Q^{a-1}X^{2a-1}
  + \O_A\left(Q^{-1-\varepsilon}\right).
\end{align*}
(Here, the assumption $Q \gg_A 1$ was used to ensure that we are in the domain of convergence of the Taylor expansion $-\log(1-x)=\sum_{i\geq1}x^i/i$.)
Thus,
\begin{align*}
&-\log(\Fcond_{K_P}(X)/H_{P,A}(X))
= R-S-\sum_{a=1}^A (-1)^a T_a \\
&= \sum_{a=1}^A \frac{(-1)^a}{a}
  \big[\underbrace{
    (X+QX^2)^a
    - (QX^2)^a
    - a X (QX^2)^{a-1}
  }_{
    = \sum_{b=2}^a \binom{a}{b} X^b (QX^2)^{a-b}
    \ll_A Q^{-1-\varepsilon}
  }\big]
  + \O_A\left(Q^{-1-\varepsilon}\right)
\ll_A Q^{-1-\varepsilon}.
\end{align*}
The claim follows by exponentiation.
\end{proof}

\begin{proof}[Proof of \Cref{thm:no-meromorphic-continuation}]
Let $A$ and $\varepsilon$ as in \Cref{lem:local-approximation}.
For $|X|<q^{-(1+\varepsilon)/2}$, the product
\[
\prod_P \frac{\Fcond_{K_P}(X^{\deg(P)})}{H_{P,A}(X^{\deg(P)})}
= \prod_{P:\ Q_P\gg_A 1} \left(1+\O_A(q^{-(1+\varepsilon)\deg(P)})\right)
\cdot \prod_{P:\ \textnormal{not } Q_P\gg_A 1} \frac{\Fcond_{K_P}(X^{\deg(P)})}{H_{P,A}(X^{\deg(P)})}
\]
has no zeros or poles. On the other hand,
\begin{equation}
\label{eq:meromorphic-approximation}
\begin{aligned}
&\prod_{P\in M_K} H_{P,A}\left(X^{\deg(P)}\right) \\
&= \prod_{P\in M_K} \frac{
    \big(1-X^{\deg(P)}\big)
    \big(1+(qX^2)^{\deg(P)}\big)
    \prod_{a=1}^A \big(1-(q^{a-1}X^{2a-1})^{\deg(P)}\big)^{(-1)^a}
  }{
    1-(q^2X^3)^{\deg(P)}
  } \\
&= \frac{
    Z_K\big(qX^2\big)
    Z_K\big(q^2X^3\big)
  }{
    Z_K\big(X\big)
    Z_K\big(q^2X^4\big)
    \prod_{a=1}^A Z_K\big(q^{a-1}X^{2a-1}\big)^{(-1)^a}
  },
\end{aligned}
\end{equation}
which is meromorphic on the entire complex plane. Hence, $\Fcond_K(X) = \prod_P \Fcond_{K_P}(X^{\deg(P)})$ can be meromorphically continued to the domain $\{|X|<q^{-(1+\varepsilon)/2}\}$ and it has the same zeros and poles in this domain as the right-hand side of \Cref{eq:meromorphic-approximation}.

The factor $Z_K(q^{a-1} X^{2a-1}) = (1-q^{a-1}X^{2a-1})^{-1} (1-q^a X^{2a-1})^{-1}$ has no zeros and has poles at exactly the $2a-1$ points $X$ with $q^{a-1} X^{2a-1}=1$ as well as the $2a-1$ points $X$ with $q^a X^{2a-1}=1$. The points with $q^{a-1} X^{2a-1}=1$ have absolute value $q^{-(a-1)/(2a-1)} > q^{-1/2}$. The points with $q^a X^{2a-1}=1$ have absolute value $q^{-a/(2a-1)} < q^{-1/2}$. The fractions $\frac{a}{2a-1}$ for $a=1,2,\dots$ are distinct since they are all reduced. The remaining (finitely many) factors on the right-hand side of \Cref{eq:meromorphic-approximation} have no zeros and just a fixed finite number of poles.

Every point $X$ on the circle $|X|=q^{-1/2}$ is an accumulation point of the set $\bigsqcup_{a\geq1}\{X\in\C : q^a X^{2a-1}=1\}$ since $q^{-a/(2a-1)}\ra q^{-1/2}$ for $a\ra\infty$. The claim follows by letting $A\ra\infty$ (and therefore $\varepsilon\ra0$).
\end{proof}

\begin{remark}
For $G=C_2$, the generating function $\Fcond_K(X)=\Fdisc_K(X)$ still happens to be rational for any rational function field $K=\F_q(T)$ of characteristic $2$.
\end{remark}

\bibliographystyle{alphaurl}
\bibliography{references.bib}

\end{document}